\documentclass[10pt]{amsart} 
\usepackage{amsmath,amsfonts,euscript,amscd,amsthm,amssymb,upref,graphics} 
\usepackage[all]{xy} 
\theoremstyle{plain} 
\newtheorem{Thm}{Theorem}[section] 
\newtheorem{Cor}[Thm]{Corollary}

\theoremstyle{definition} 
\newtheorem{Def}[Thm]{Definition}

\newtheorem*{theorem*}{Theorem}
\theoremstyle{remark} 
\newtheorem{Rem}[Thm]{Remark} 
\newtheorem{Exa}[Thm]{Example} 
\newcommand{\C}{{\mathbb C}} 
\begin{document}

\title[Property (T) for the special linear group of holomorphic functions]{On Kazhdan's Property (T) for the special linear group of holomorphic functions} 
\author{Bj\"orn Ivarsson and Frank Kutzschebauch}  
\address{Department of Natural Sciences, Engineering and Mathematics, Mid Sweden University, SE-851 70 Sundsvall, Sweden}
\email{Bjorn.Ivarsson@miun.se}
\address{Institute of Mathematics, University of Bern, Sidlerstrasse 5, CH-3012 Bern, Switzerland}
\email{frank.kutzschebauch@math.unibe.ch} 
\thanks{{\bf Acknowledgements:}  The research of the
second author was  partially supported by Schweizerische
Nationalfonds grant No. 200020-134876/1}
{\renewcommand{\thefootnote}{} }
\date{\today} \setcounter{tocdepth}{2}

\begin{abstract}
	We investigate when  the group  $\mbox{SL}_n(\mathcal{O}(X))$ of holomorphic maps from a Stein space $X$ to $\mbox{SL}_n (\C)$ has Kazhdan's property (T) for $n\ge 3$. This provides a new class of examples of non-locally compact groups having Kazhdan's property (T). In particular we prove that the holomorphic loop group of  $\mbox{SL}_n (\C)$ has Kazhdan's property (T) for $n\ge 3$. 
Our result relies on the method of Shalom to prove Kazhdan's property (T) and the solution to  Gromov's Vaserstein problem by the authors.
\end{abstract}
\maketitle 
\bibliographystyle{alpha} 
\tableofcontents 
\section{Introduction} Kazhdan's property (T) expresses a certain rigidity for unitary Hilbert space representations of a topological group $G$. Even though it was invented by Kazhdan for discrete or more generally locally compact groups Kazhdan's property (T)  has turned up in so many different aspects of mathematics that its study in any context seems of  interest. The first examples of non-locally compact groups with this property were construced by Shalom in \cite{ShalomBGKPT}. We find it  interesting that also holomorphic objects can have it. Our main result, in particular showing that the holomorphic loop group  of  $\mbox{SL}_n (\C)$ has Kazhdan's property (T) for $n\ge 3$, is the following

\begin{theorem*} (see Theorem \ref{KPT})
 Let $n\ge 3$ and $X$ be a  Stein manifold having finitely many connected components  and with the property that all holomorphic maps from $X$ to $\mbox{SL}_n(\C)$ are null-homotopic. Then $\mbox{SL}_n(\mathcal{O}(X))$ has Kazhdan's property (T). 
\end{theorem*}
Some generalization  to the case when not all  holomorphic maps from $X$ to $\mbox{SL}_n(\C)$ are null-homotopic is contained in the last section.

 Let's recall the definitions:   By $\mathcal{O}(X)$ we denote the algebra of holomorphic functions on $X$ endowed with compact-open topology.
By Stein's original definition, simplified by later developments,  a complex manifold $X$ is Stein (or, as Karl Stein put it, holomorphically complete) if it satisfies the following two conditions.

\smallskip\noindent
(1) Holomorphic functions on $X$ separate points, that is, if  $x,y \in X,  x\ne y$, then there is $f\in \mathcal{O}(X)$  such that $f(x) \ne  f(y)$. 

\smallskip\noindent
(2) $X$ is holomorphically convex, that is, if $K\subset X$ is compact, then its $\mathcal{O}(X))$-hull $\hat K$, consisting of all $x \in X$  with $ \vert f(x)\vert  \le {\rm max}_K \vert f\vert$  for all $f \in \mathcal{O}(X)$, is also compact. Equivalently, if $ E \subset X$  is not relatively compact, then there is $ \in \mathcal{O}(X)$  such that $f\vert_E$ is unbounded.

A domain in $\C^n$ is Stein if and only if it is a domain of holomorphy. Every noncompact Riemann surface is Stein.
By the Remmert embedding theorem (see also Remark \ref{Stein,numbers})  a connected complex manifold is Stein if and only if it is biholomorphic to a closed complex submanifold of $C^N$ for some $N$. If $X$ is  not smooth, i.e., a complex space, the notion of Stein space is defined by the same two conditions (1) and (2). A connected Stein space
is biholomorphic to some analytic subspace of $\C^n$ if and only if it has an upper bound on the dimension of its tangent spaces. We refer to \cite{For3} for more information on Stein manifolds.

Here comes the  definition of property (T), for a comprehensive introduction to Kazhdan's property (T) we  refer to \cite{BachirDeLaHarpeValetteKPT}.

\begin{Def}
	Let $G$ be a topological group, $K\subset G$ a subset, $\varepsilon > 0$, $H$ a Hilbert space, and $(\pi,H)$ a continuous unitary $G$-representation. A vector $v\in H$ is called $(K,\varepsilon)$-invariant if $\| \pi(g)v-v\|<\varepsilon \|v\|$ for all $g\in K$. 
\end{Def}
\begin{Def}
	A topological group $G$ has Kazhdan's property (T) (or is a Kazhdan group) if there is a compact $K\subset G$ and $\varepsilon > 0$ such that every continuous unitary $G$-representation with a $(K,\varepsilon)$-invariant vector contains a non-zero $G$-invariant vector. We call $(K,\varepsilon)$ a {\it Kazhdan pair} for $G$ and $\varepsilon$ is called a {\it Kazhdan constant} for $K$ and $G$. 
\end{Def}

We like to thank Pierre de la Harpe for inspiring discussions on the subject and for bringing the paper of Cornulier \cite{dCKPSCF} to our attention.

\section{Kazhdan's property (T) and the special linear group of holomorphic functions}

The present investigations rely on results by Shalom \cite{ShalomBGKPT} and the authors \cite{IvarssonKutzschebauchHFMSLG}, see also \cite{IvarssonKutzschebauchSGVP}. 
\begin{Def}
	Let $R$ be a commutative ring with unit. The group $\mbox{SL}_n(R)$ is said to be boundedly elementary generated if there is a $\nu<\infty$ such that every matrix in $\mbox{SL}_n(R)$ can be written as a product of at most $\nu=\nu_n(R)$ elementary matrices. 
\end{Def}
Recall that an elementary matrix in $\mbox{SL}_n(R)$ is a matrix of the form $I+ rE_{i,j}$ $i \ne j$, {\it i.e}, with ones on the diagonal and all entries outside the diagonal zero except for one entry.

The following is a result of Shalom from \cite{ShalomBGKPT}, see also \cite[Theorem 4.3.5]{BachirDeLaHarpeValetteKPT}, 
\begin{Thm}
	\label{shalom} Let $n\ge 3$ and $R$ be a topological commutative ring with unit. Assume that $\mbox{SL}_n(R)$ is boundedly elementary generated and that there is a finite set $\{\alpha_1,\dots,\alpha_m\}\subset R$ generating a dense subring of $R$. Then $\mbox{SL}_n(R)$ has Kazhdan's property (T). 
\end{Thm}

\begin{Rem}
	\label{Kconst} More precisely if there are $m$ elements in $R$ generating a dense subring then there is a compact $K$ in $\mbox{SL}_n(R)$ such that $\varepsilon=1/(22^{m+1}\nu_n(R))$ is a Kazhdan constant for $K$ and $\mbox{SL}_n(R)$, see \cite[Remark 4.3.6]{BachirDeLaHarpeValetteKPT}. 
\end{Rem}

In this paper we consider the ring $R = \mbox{SL}_n(\mathcal{O}(X))$ of holomorphic maps from a Stein space $X$ to $\mbox{SL}_n(\mathbb{C})$ endowed with compact-open topology, the natural topology for holomorphic mappings.
Using the above terminology the authors show in \cite{IvarssonKutzschebauchHFMSLG} that $\mbox{SL}_n(\mathcal{O}(X))$ is boundedly elementary generated when $X$ is a finite dimensional reduced Stein space with the property that all holomorphic mappings from $X$ to $\mbox{SL}_n(\C)$ are null-homotopic, that is homotopic (through a family of continuous maps) to a constant map.  This result is an application of the Oka-Grauert-Gromov-h-principle in Complex Analysis. For more information on that important principle in Complex Analytic Geometry we refer the interested reader to the monograph of Forstneri\v c \cite{For3}. The precise statement of the theorem proved in \cite{IvarssonKutzschebauchHFMSLG} is the following:
\begin{Thm}
	\label{IK} Let $X$ be a finite dimensional reduced Stein space and $f\colon X\to \mbox{SL}_n(\mathbb{C})$ be a holomorphic mapping that is null-homotopic. Then there exist a natural number $K$, depending only on the dimension of $X$ and $n$, and holomorphic mappings $G_1,\dots, G_{K}\colon X\to \mathbb{C}^{m(m-1)/2}$ such that $f$ can be written as a product of upper and lower diagonal unipotent matrices 
	\begin{equation*}
		f(x) = \left( 
		\begin{matrix}
			1 & 0 \cr G_1(x) & 1 \cr 
		\end{matrix}
		\right) \left( 
		\begin{matrix}
			1 & G_2(x) \cr 0 & 1 \cr 
		\end{matrix}
		\right) \ldots \left( 
		\begin{matrix}
			1 & G_K(x)\cr 0 & 1 \cr 
		\end{matrix}
		\right). 
	\end{equation*}
\end{Thm}
\begin{Rem}
	It is a consequence of the classical Oka-Grauert principle that the homotopy to the constant map can be chosen through a family of holomorphic maps. Therefore null-homotopic in the topological sense and holomorphic sense is equivalent in this setting. 
\end{Rem}

Combining Shalom's and the authors' result we get the following new examples of non-locally compact Kazhdan groups. For similar results for the ring of continuous functions on a finite dimensional topological space see \cite{dCKPSCF}. 
\begin{Thm}
	\label{KPT} Let $n\ge 3$ and $X$ be a  Stein manifold with finitely many connected components and with the property that all holomorphic maps from $X$ to $\mbox{SL}_n(\C)$ are null-homotopic. Then $\mbox{SL}_n(\mathcal{O}(X))$ has Kazhdan's property (T). 
\end{Thm}
\begin{proof}
	A finite set of functions that generate a dense subring of $\mathcal{O}(X)$ can constructed as follows. Embed $X$ into $\C^N$ using Remmert's embedding theorem. By the Oka-Weil theorem $\C[z_1,\dots,z_N]|_X$ is dense in $\mathcal{O}(X)$. Finally we see that the set of functions $S=\{z_1,\dots,z_N,\sqrt{2},i\}$ generates a dense subring of $\C[z_1,\dots,z_N]$. Therefore $S$ also generates a dense subring of $\mathcal{O}(X)$. 
	
	Every unipotent matrix can be written as a product of $n(n-1)/2$ elementary matrices so $\mbox{SL}_n(\mathcal{O}(X))$ is boundedly elementary generated by Theorem \ref{IK} and by Theorem \ref{shalom} the group $\mbox{SL}_n(\mathcal{O}(X))$ is a Kazhdan group. 
\end{proof}

\begin{Rem}\label{Stein,numbers}
	Let $N$ be the smallest dimension such that $X$ embeds into $\C^N$. By work of Gromov, Eliashberg and Sch\"urmann $N$ for a connected $X$ is bounded by $\lfloor 3(\dim X)/2\rfloor + 1$ if $\dim X \ge 2$, see \cite{EliashbergESMAS} and \cite{SchurmannESSMD}. We have that $\varepsilon=1/(22^{N+3}\nu_n(\mathcal{O}(X)))$ is a Kazhdan constant by Remark \ref{Kconst}. Therefore both the minimal embedding dimension and the number of elementary matrices needed to factorize a null-homotopic holomorphic map $f\colon X \to \mbox{SL}_n(\C)$ is of great interest. The study of minimal embedding dimension is a classical difficult subject in complex analysis. The numbers $\nu_n(\mathcal{O}(X))$ as well as the corresponding numbers $\nu_n(C(T))$ (existing by work of Vaserstein \cite{VasersteinRMDPDFAO}) for the ring of continuous functions on a finite dimensional topological space $T$ are mostly unknown. A first study of these numbers in the holomorphic case can be found in \cite{IvarssonKutzschebauchNF}. 
\end{Rem}
\begin{Rem}
	Theorem \ref{KPT} also holds for reduced Stein spaces that can be embedded in some $\C^N$. The proof is exactly the same once we have the embedding. For more on the embedding dimension for Stein spaces see \cite{SchurmannESSMD}.
\end{Rem}
\begin{Rem}
	By \cite[Example 1.7.4 (iv)]{BachirDeLaHarpeValetteKPT} $\mbox{SL}_2(\C)$ is not a Kazhdan group. Since the closure of the image of a Kazhdan group under a continuous homomorphism is again a Kazhdan group, see \cite[Theorem 1.3.4]{BachirDeLaHarpeValetteKPT}, it follows that $\mbox{SL}_2(\mathcal{O}(X))$ never has Kazhdan's property (T). 
\end{Rem}
\begin{Cor}
	Let $n\ge 3$ and $X$ be a contractible Stein manifold. Then $\mbox{SL}_n(\mathcal{O}(X))$ has Kazhdan's property (T). 
\end{Cor}
Since $\mbox{SL}_n(\C)$ is simply connected we get the following. 
\begin{Cor}
	For $n\ge 3$ the group $\mbox{SL}_n(\mathcal{O}(\C^*))$, {\it i.e.} the holomorphic loop group of $\mbox{SL}_n(\C)$, is a Kazhdan group. 
\end{Cor}

\section{A generalization} It is natural to ask what can be said when $X$ is a Stein manifold having holomorphic maps into $\mbox{SL}_n(\C)$ that are not null-homotopic. In this case Theorem \ref{IK} still can be applied to get some information. Let $\mbox{E}_n(\mathcal{O}(X))$ denote the group generated by the elementary matrices in $\mbox{SL}_n(\mathcal{O}(X))$. In general we get the following result as consequence of Theorem \ref{IK}. 
\begin{Thm}
	Let $n\ge 3$ and $X$ be a finite dimensional Stein space. Then $\mbox{E}_n(\mathcal{O}(X))$ has Kazhdan's property (T). 
\end{Thm}
First note that all matrices in $\mbox{E}_n(\mathcal{O}(X))$ are null-homotopic. Now the proof is exactly the same as the proof of Theorem \ref{KPT}. The point where we use Theorem \ref{IK} is to conclude that $\mbox{E}_n(\mathcal{O}(X))$ is boundedly elementary generated.

We now have the following. 
\begin{Thm}\label{steinkpt}
	Let $n\ge 3$ and $X$ be a finite dimensional Stein space with finite embedding dimension. Then $\mbox{SL}_n(\mathcal{O}(X))$ has Kazhdan's property (T) if and only if $$\mbox{SL}_n(\mathcal{O}(X))/\overline{\mbox{E}_n(\mathcal{O}(X))}$$ has Kazhdan's property (T). 
\end{Thm}
\begin{proof}
	Let $H$ be a subgroup of a topological group $G$. It is easy to check from the definition that $\overline{H}$ is a Kazhdan group when $H$ is. The result now follows from \cite[Proposition 1.7.6 and Remark 1.7.9]{BachirDeLaHarpeValetteKPT} which says that a Frechet group $G$ has Kazhdan's property (T) if, for a normal closed subgroup $N$, both $G/N$ and $N$ has. The other implication follows from \cite[Theorem 1.3.4]{BachirDeLaHarpeValetteKPT} which says that Kazhdan's property (T) is inherited by quotients. 
\end{proof}
\begin{Rem}
	We don't know if it always holds that $\mbox{E}_n(\mathcal{O}(X))= \overline{\mbox{E}_n(\mathcal{O}(X))}$ when we equip $\mbox{SL}_n(\mathcal{O}(X))$ with the compact-open topology. A Stein manifold is homotopy equivalent to a finite dimensional CW-complex and when this complex is finite then $\mbox{E}_n(\mathcal{O}(X))= \overline{\mbox{E}_n(\mathcal{O}(X))}$. However the complex can be infinite and we believe that there are examples where $\mbox{E}_n(\mathcal{O}(X)) \neq \overline{\mbox{E}_n(\mathcal{O}(X))}$. 
\end{Rem}
\begin{Exa}
	As an application of Theorem \ref{steinkpt} we study the quadrics $Q_k=\{z\in \C^{k+1}; z_1^2 + \cdots + z_{k+1}^2 = 1\}$, $k\ge 1$. We claim that $Q_k$ is homotopy equivalent to the $k$-dimensional sphere $S^k$. Indeed $Q_k$ is isomorphic to the homogeneous space $\mbox{SO}_{k+1}(\C)/\mbox{SO}_{k}(\C)$, that is the quotient of the two reductive groups $\mbox{SO}_{k+1}(\C)$ and $\mbox{SO}_{k}(\C)$. By the Mostow decomposition theorem, see \cite{MostowCFKS,MostowSNDTSSG} or \cite[Section 3.1]{HeinznerGITSS}, a homogeneous space of complex reductive Lie groups $K^{\C}/L^{\C}$ admits a strong deformation retraction onto the quotient of their maximal compact subgroups $K/L$. In our case this quotient is $\mbox{SO}_{k+1}(\mathbb{R})/\mbox{SO}_{k}(\mathbb{R})$ which is isomorphic to $S^k$.
	
	We have $$\mbox{SL}_n(\mathcal{O}(Q_k))/\mbox{E}_n(\mathcal{O}(Q_k))=\pi_k(\mbox{SL}_n(\C)).$$ Let $n\ge 3$. The homotopy groups $\pi_k(\mbox{SL}_n(\C))$ are abelian. Discrete groups that are Kazhdan groups have finite abelianization \cite[Corollary 1.3.6]{BachirDeLaHarpeValetteKPT} and finite groups are Kazhdan groups. Therefore $\mbox{SL}_n(\mathcal{O}(Q_k))$ is a Kazhdan group precisely when $\pi_k(\mbox{SL}_n(\C))$ is finite. The homotopy groups of $\mbox{SL}_n(\C)$ are known to be infinite precisely when $k$ is odd and $3\le k \le 2n-1$, see for example \cite{MimuraTodaTLG}. Therefore $\mbox{SL}_n(\mathcal{O}(Q_k))$ is a Kazhdan group if and only if $k=1$, $k$ even, or $k\ge 2n$. 
\end{Exa}


\begin{thebibliography}{BdlHV08}

\bibitem[BdlHV08]{BachirDeLaHarpeValetteKPT}
Bachir Bekka, Pierre de~la Harpe, and Alain Valette.
\newblock {\em Kazhdan's property ({T})}, volume~11 of {\em New Mathematical
  Monographs}.
\newblock Cambridge University Press, Cambridge, 2008.

\bibitem[dC06]{dCKPSCF}
Yves de~Cornulier.
\newblock Kazhdan property for spaces of continuous functions.
\newblock {\em Bull. Belg. Math. Soc. Simon Stevin}, 13(5):899--902, 2006.

\bibitem[EG92]{EliashbergESMAS}
Yakov Eliashberg and Mikhael Gromov.
\newblock Embeddings of {S}tein manifolds of dimension {$n$} into the affine
  space of dimension {$3n/2+1$}.
\newblock {\em Ann. of Math. (2)}, 136(1):123--135, 1992.

\bibitem[Fo11]{For3} F.\ Forstneri\v{c}.
\newblock{\em Stein manifolds and holomorphic mappings}, Ergeb.\ Math.\
Grenzgeb.\ (3), vol.\ 56, 
\newblock Springer-Verlag, 2011.

\bibitem[Hei91]{HeinznerGITSS}
Peter Heinzner.
\newblock Geometric invariant theory on {S}tein spaces.
\newblock {\em Math. Ann.}, 289(4):631--662, 1991.

\bibitem[IK08]{IvarssonKutzschebauchSGVP}
Bj{\"o}rn Ivarsson and Frank Kutzschebauch.
\newblock A solution of {G}romov's {V}aserstein problem.
\newblock {\em C. R. Math. Acad. Sci. Paris}, 346(23-24):1239--1243, 2008.

\bibitem[IK12a]{IvarssonKutzschebauchHFMSLG}
Bj{\"o}rn Ivarsson and Frank Kutzschebauch.
\newblock Holomorphic factorization of mappings into $\mbox{SL}_n(\mathbb{C})$.
\newblock {\em Ann. of Math. (2)}, 175:45--69, 2012.

\bibitem[IK12b]{IvarssonKutzschebauchNF}
Bj{\"o}rn Ivarsson and Frank Kutzschebauch.
\newblock On the number of factors in the unipotent factorization of
  holomorphic mappings into $\mbox{SL}_2(\mathbb{C})$.
\newblock {\em Proc. Amer. Math. Soc.}, 140:823--838, 2012.

\bibitem[Mos55a]{MostowCFKS}
G.~D. Mostow.
\newblock On covariant fiberings of {K}lein spaces.
\newblock {\em Amer. J. Math.}, 77:247--278, 1955.

\bibitem[Mos55b]{MostowSNDTSSG}
G.~D. Mostow.
\newblock Some new decomposition theorems for semi-simple groups.
\newblock {\em Mem. Amer. Math. Soc.}, 1955(14):31--54, 1955.

\bibitem[MT91]{MimuraTodaTLG}
Mamoru Mimura and Hirosi Toda.
\newblock {\em Topology of {L}ie groups. {I}, {II}}, volume~91 of {\em
  Translations of Mathematical Monographs}.
\newblock American Mathematical Society, Providence, RI, 1991.
\newblock Translated from the 1978 Japanese edition by the authors.

\bibitem[Sch97]{SchurmannESSMD}
J.~Sch{\"u}rmann.
\newblock Embeddings of {S}tein spaces into affine spaces of minimal dimension.
\newblock {\em Math. Ann.}, 307(3):381--399, 1997.

\bibitem[Sha99]{ShalomBGKPT}
Yehuda Shalom.
\newblock Bounded generation and {K}azhdan's property ({T}).
\newblock {\em Inst. Hautes \'Etudes Sci. Publ. Math.}, (90):145--168 (2001),
  1999.

\bibitem[Vas88]{VasersteinRMDPDFAO}
L.~N. Vaserstein.
\newblock Reduction of a matrix depending on parameters to a diagonal form by
  addition operations.
\newblock {\em Proc. Amer. Math. Soc.}, 103(3):741--746, 1988.

\end{thebibliography}

\end{document}